\documentclass{article}
\usepackage{amsmath}
\usepackage{amssymb}
\usepackage{amsthm}
\usepackage{graphicx}
\usepackage{bbm}
\usepackage{enumerate}
\usepackage[colorlinks=true,urlcolor=blue,linkcolor=blue,plainpages=false,pdfpagelabels]{hyperref}

\newtheorem{theorem}{Theorem}[section]
\newtheorem{proposition}[theorem]{Proposition}
\newtheorem{lemma}[theorem]{Lemma}
\newtheorem{corollary}[theorem]{Corollary}

\newtheorem{definition}[theorem]{Definition}

\newcommand{\sxn}{(x_n)}
\newcommand{\sfxn}{(f(x_n))}
\newcommand{\N}{\mathbb{N}}
\newcommand{\R}{\mathbb{R}}
\newcommand{\resgf}{J_f^\gamma}
\newcommand{\resuf}{J_f^1}
\newcommand{\beq}{\begin{equation}} \newcommand{\eeq}{\end{equation}}
\newcommand{\be}{\begin{enumerate}} \newcommand{\ee}{\end{enumerate}}
\newcommand {\bua} {\begin{eqnarray*}}
\newcommand {\eua} {\end {eqnarray*}}
\newcommand {\tilchi}{\tilde{\chi}}

\title{An application of proof mining to the proximal point algorithm in CAT(0) spaces}
\author{Lauren\c tiu Leu\c stean${}^{a,b}$ and Andrei Sipo\c s${}^{a,b}$\\[0.2cm]
\footnotesize ${}^a$Faculty of Mathematics and Computer Science, University of Bucharest,\\
\footnotesize Academiei 14, 010014 Bucharest, Romania\\[0.1cm]
\footnotesize ${}^b$Simion Stoilow Institute of Mathematics of the Romanian Academy,\\
\footnotesize P. O. Box 1-764, 014700 Bucharest, Romania\\[0.1cm]
\footnotesize E-mails: laurentiu.leustean@unibuc.ro, Andrei.Sipos@imar.ro\\
}
\date{}

\begin{document}

\maketitle
\begin{center}
{\em  Dedicated to the memory of Professor Solomon Marcus (1925-2016)}
\end{center}

\vspace*{2mm}

\begin{abstract}
\noindent We compute, using techniques originally introduced by Kohlenbach, the first author and Nicolae, uniform rates of metastability for the proximal point algorithm in the context of CAT(0) spaces (as first considered by Ba\v{c}\'ak), specifically for the case where the ambient space is totally bounded. This result is part of the program of proof mining, which aims to apply methods of mathematical logic with the purpose of extracting quantitative information out of ordinary mathematical proofs, which may not be necessarily constructive.\\


\end{abstract}

\section{Introduction}

The proximal point algorithm is a fundamental tool of convex optimization, going back to Martinet 
\cite{Mar70}, Rockafellar \cite{Roc76} and Br\'ezis and Lions \cite{BreLio78}. Since its 
inception, the schema turned out to be highly versatile, covering in its various developments, 
{\it inter alia}, the problems of finding zeros of monotone operators, minima of convex 
functions and fixed points of nonexpansive mappings. For a general introduction to the field 
in the context of Hilbert spaces, see the book of Bauschke and Combettes \cite{BauCom10}.

A recent breakthrough was achieved by Ba\v{c}\'ak \cite{Bac13}, who proved the weak convergence 
in complete CAT(0) spaces (that is, $\Delta$-convergence) of the variant of the algorithm used to 
find minima of convex, lower semicontinuous (lsc) proper functions. Let us detail the statement of his result. 
If $X$ is a complete CAT(0) space  and  $f: X \to (-\infty, +\infty]$  is a convex, lsc proper function that has minimizers,  
then, following Jost \cite{Jos95}, we may define its resolvent by the relation
$$J_f(x) := \displaystyle{\arg\!\min}_{y \in X} \left[ f(y) + \frac12 d^2(x,y)\right].$$
For such an $f$, a starting point $x \in X$, and a sequence of weights 
$(\gamma_n)_{n \in \mathbb{N}}$, the proximal point algorithm  $(x_n)_{n \in \mathbb{N}}$ is defined by setting:
\begin{center} $x_0:=x$, \quad  $x_{n+1} := J_{\gamma_n f} x_n$  for any $n\in\N$. 
\end{center}
Ba\v{c}\'ak's result then states that, 
conditional on the fact that  $\sum_{n=0}^\infty \gamma_n = \infty$, the sequence $(x_n)$  
converges weakly to a minimizer of $f$. As a consequence, one gets (see \cite[Remark 1.7]{Bac13})

\begin{theorem}\label{ppa-cat0-strong}
In the above hypotheses, assume, furthermore, that $X$ is a complete  locally compact CAT(0) 
space. Then  $\sxn$ converges strongly to a a minimizer of $f$. 
\end{theorem}

The proof of Theorem \ref{ppa-cat0-strong} is what we are going to build upon, roughly, in our quantitative analysis 
from the viewpoint of proof mining.

Proof mining is a subfield of applied logic that seeks to use proof interpretations, 
like G\"odel's Dialectica or functional interpretation \cite{Goe58}, 
originally developed with the purpose of giving consistency 
arguments for systems of arithmetic, in order to extract quantitative information out 
of proofs in ordinary mathematics. Under the name of ``proof unwinding'', it was first 
proposed as a viable research program by G. Kreisel in the 1950s and after several decades 
of sporadic advances (one of the most significant being H. Luckhardt's 1989 analysis \cite{Luc89} 
of the proof of Roth's theorem on diophantine approximations) it was given maturity in the 1990s 
and the 2000s by U. Kohlenbach and his collaborators. The project has culminated into  
the  general logical metatheorems developed by Kohlenbach \cite{Koh05} and by 
Gerhardy and Kohlenbach \cite{GerKoh08} for proofs in metric, (uniformly convex) normed and 
inner product spaces, as well as geodesic spaces like $W$-hyperbolic spaces and CAT(0) spaces.
These logical metatheorems were extended to other classes of fundamental spaces in nonlinear and 
functional analysis, optimization, geometric group theory and geodesic geometry: 
Gromov hyperbolic spaces, $\R$-trees and a class of uniformly convex geodesic spaces \cite{Leu06}, 
completions of metric and normed spaces \cite{Koh08}, totally bounded metric spaces \cite{Koh08,KohLeuNicXX}, 
uniformly smooth normed spaces \cite{KohLeu12},  Banach lattices and $C(K)$ spaces \cite{GunKoh16}, 
$L^p$-spaces \cite{GunKoh16,SipXX} and CAT($\kappa$) spaces \cite{KohNic17}.
These logical metatheorems guarantee that from proofs of $\forall\exists$-sentences (satisfying some conditions) in formal systems associated to 
such abstract spaces $X$,
one can extract effective uniform bounds on existentially quantified variables.
Kohlenbach's monograph from 2008 \cite{Koh08} covers the major results in the field until then, while a 
survey of recent developments is \cite{Koh16}.

The canonical example of an existentially quantified variable in ordinary mathematics comes from the definition 
of the limit of a sequence in a metric space $(X,d)$. If $\sxn$ is a sequence in $X$ and $x\in X$, then 
$\lim_{n\to\infty}x_n=x$ if and only if

$$\forall k \in {\mathbb N} \,\exists N \in {\mathbb N} \,\forall n \geq N\, \left(d(x_n,x) 
\leq \frac1{k+1}\right).$$

A witness for this existentially quantified $N$, also called {\it rate of convergence} 
for the sequence, as it will be defined in more detail further below, would consist of a 
formula giving it in terms of the $k$. Unfortunately, as the sentence above has three 
alternating quantifiers in a row (i.e. $\forall\exists\forall$), the techniques of proof mining 
preclude the extraction of such a computable rate if the proof is non-constructive in 
the sense of using at least once the law of excluded middle (one can show that the existence 
of a general procedure for these cases would contradict the impossibility of the halting problem). 
Four avenues have generally been tried so far in proof mining, if the convergence of a sequence 
was under discussion. The first one is the extraction of the full rate of convergence in the 
rare case that the proof is fully or at least partially constructive. The second one is to 
settle for a weaker property, like the limit inferior, which may have a tractable 
$\forall\exists$ form (and if the sequence is nonincreasing, the extracted 
modulus of liminf would also be a rate of convergence). The third one is to use some uniqueness 
properties of the limit in order to extract the rate of convergence from a distantly related property 
like the rate of asymptotic regularity. Finally, the fourth way is what we are going to focus on here. 
It consists of considering instead of convergence  the Cauchy property of the sequence 
$$\forall k \in {\mathbb N} \,\exists N \in {\mathbb N} \,\forall p\in {\mathbb N}\, 
\left(d(x_N,x_{N+p})\leq \frac1{k+1}\right)$$
and replacing it with an equivalent formulation (known in logic as its Herbrand normal form or 
its Kreisel no-counterexample interpretation),  called {\em metastability} by Tao \cite{Tao07,Tao08}.
The following sentence expresses the metastability of the sequence 
above:
$$\forall k \in {\mathbb N} \,\forall g: {\mathbb N} \to {\mathbb N} \,\exists N \in {\mathbb N} 
\,\forall i,j \in [N, N+g(N)]\,\,\, \left(d(x_i,x_j) \leq \frac1{k+1}\right).$$
It is immediately seen that this sentence is of a reduced $\forall\exists$ logical complexity. 
It is, however, a simple exercise, to check that it is classically (but not intuitionistically) 
equivalent to the assertion that the sequence under discussion is Cauchy. Therefore, one can 
now say that the fourth way is focused on obtaining a {\it rate of metastability} for the 
sequence, i.e. a mapping $\Phi:\N\times\N^\N\to \N$ satisfying, for all $k\in\N$ and all $g: {\mathbb N} \to {\mathbb N}$,
\beq
\exists N \leq \Phi(k,g)\,\forall i,j \in [N, N+g(N)]\,\,\, \left(d(x_i,x_j) \leq \frac1{k+1}\right).
\eeq

In a recent paper, Kohlenbach, the first author and Nicolae \cite{KohLeuNicXX} have 
studied a general line of argument used in convergence proofs in nonlinear analysis and convex 
optimization. Specifically, it is often the case that an iterative sequence is proven to be 
convergent to a point in a certain set $F$ (e.g. the set of fixed points of an operator using 
which the sequence was constructed) if it sits inside a compact space, it is {\em Fej\'er monotone}
with respect to $F$ (that is, for all $q \in F$ and all $n \in \mathbb{N}$, $d(x_{n+1},q)\leq d(x_n,q)$) 
and it has ``approximate $F$-points'', i.e. points which are, in a sense, near $F$. 
The main result in  \cite{KohLeuNicXX}  is that all this can be made effective. For that to work, however, 
the three hypotheses must also be transformed into a quantitative form. A ``modulus of total boundedness" 
witnesses the space being compact. For the other two properties, one must formulate what exactly does 
it mean for a point to be ``near'' to $F$. This is done in terms of an approximation 
$F = \bigcap_{k \in \mathbb{N}} AF_k$, which helps formulate both the ``modulus of uniform 
Fej\'er monotonicity'' and the ``approximate $F$-point bound''. The choice of an approximation 
to $F$, as well as the computation of these moduli, has been done in \cite{KohLeuNicXX,LeuRadSip16,KohLopNic17} 
for some classical iterations associated to important classes of mappings and operators.

\mbox{}

In this paper we apply the techniques developed in \cite{KohLeuNicXX} to obtain a quantitative version 
of Theorem \ref{ppa-cat0-strong}, providing an effective uniform rate of metastability for 
the proximal point algorithm in totally bounded CAT(0) spaces. The next section will give some 
preliminaries on the proximal point algorithm, while the last section of the paper is dedicated to 
the proof of our main quantitative result, Theorem \ref{main-quant-thm}.

We finish this Introduction with a recall of definitions from \cite{KohLeuNicXX} and quantitative 
notions that will be used throughout the paper. We point out, first, that  ${\mathbb N}=\{0,1,2,\ldots\}$ and 
that we denote $[m,n]=\{m,m+1,\ldots,n\}$ for any $m, n\in{\mathbb N}$ with $m\leq n$. 

Let $(X,d)$ be a metric space. For any mapping $T:X\to X$ we denote by $Fix(T)$ the set of fixed points of $T$.

 A {\it modulus of total boundedness} for  $X$ is a function $\alpha : \mathbb{N} \to \mathbb{N}$ 
such that for any $k \in \mathbb{N}$ and any sequence $\sxn$ in $X$ there exist $i<j$ in $[0,\alpha(k)]$ such that
$$d(x_i,x_j) \leq \frac1{k+1}.$$
This notion was first used in \cite{Ger08} to analyze, using proof mining methods, the 
Furstenberg-Weiss proof of the Multiple Birkhoff Recurrence Theorem. One can easily see 
that $X$ is totally bounded if and only if it has a modulus 
of total boundedness.

Let $F\subseteq X$.  We say that a family $(AF_k)_{k \in \mathbb{N}}$ of subsets of $X$ is an {\it approximation} to  $F$ if
$$F = \bigcap_{k \in \mathbb{N}} AF_k \quad \text{ and } AF_{k+1} \subseteq AF_k \text{~for all~} k\in\N.$$
Elements of $AF_k$ are also called $k$-approximate $F$-points.

\begin{definition}\cite{KohLeuNicXX}
Let $F\subseteq X$ be a set with an approximation $(AF_k)$.
\begin{enumerate}[(i)]
\item $F$ is  {\it uniformly closed} with respect to  $(AF_k)$ with moduli $\delta_F, \omega_F : \mathbb{N} \to \mathbb{N}$ if for all $k \in \mathbb{N}$ and 
all $p, q \in X$ we have that 
\begin{center}
$q \in  AF_{\delta_F(k)}$ and $\displaystyle d(p,q) \leq \frac1{\omega_F(k)+1}$ \quad imply\quad  $p \in AF_k.$
\end{center}
\item $\sxn$ is  {\it uniformly Fej\'er monotone} with respect to $(AF_k)$ with modulus $\chi$ if for all $n,m,r \in \mathbb{N}$, all $p \in AF_{\chi(n,m,r)}$ and all $l \leq m$ we have that
$$d(x_{n+l},p) < d(x_n,p) + \frac1{r+1}.$$
\item $\sxn$ has {\it  approximate $F$-points} with respect to $(AF_k)$ with modulus $\Phi$ (which is taken to be nondecreasing) if for all $k \in \mathbb{N}$ there is an $N \leq \Phi(k)$ such that $x_N \in AF_k$.
\end{enumerate}
\end{definition}
We refer to \cite[Sections 3 and 4]{KohLeuNicXX} for details and intuitions behind the above definitions. 
We remark that one can get  nondecreasing moduli using the following  transformation. For any $f:\mathbb{N} \to \mathbb{N}$, one defines $f^M: \mathbb{N} \to \mathbb{N}$ by
$f^M(n) := \max_{i \leq n} f(i).$
Then $f^M$ is nondecreasing and for any $n$, we have that $f(n) \leq f^M(n)$.

We now give some notions that are customary in quantitatively expressing some basic properties 
of real-valued sequences. Let $(a_n)_{n\in{\mathbb N}}$ be a sequence of nonnegative real 
numbers. If $(a_n)$ converges to $0$, then a {\it rate of convergence} for $(a_n)$ is a mapping 
$\beta:{\mathbb N}\to{\mathbb N}$ such that  for all $k\in\N$,
$$\forall n\geq \beta(k)\,\,\,  \left(a_n \leq \frac1{k+1}\right).$$
If the series $\sum\limits_{n=0}^\infty a_n$ diverges, then a function $\theta:{\mathbb N}\to{\mathbb N}$ is called a {\it rate of divergence} of the series if for all $P \in {\mathbb N}$ we have that
$$\sum_{n=0}^{\theta(P)}a_n \geq P.$$
A {\it modulus of liminf} for $(a_n)$ is a mapping $\Delta:{\mathbb N}\times {\mathbb N}\to{\mathbb N}$, satisfying, for all $k,L\in\N$,
$$ \exists N \in [L,\Delta(k,L)] \,\,\,\,  \left( a_N  \leq \frac1{k+1}\right).$$
Such a modulus exists if and only if $\displaystyle \liminf_{n\to\infty} a_n=0$.

More generally, let $(b_n)_{n\in{\mathbb N}}$ be a sequence of real numbers and $b\in\R$. 
If  $(b_n)$ converges to $b$, then a {\it a rate of convergence} of $(b_n)$ is a mapping 
$\beta:{\mathbb N}\to{\mathbb N}$ such that  for all $k\in\N$,
\beq
\forall n\geq \beta(k)\,\,\,  \left(|b_n-b|\leq \frac1{k+1}\right).\label{def-rate-conv}
\eeq
Thus, a rate of convergence of $(b_n)$  coincides with a rate of convergence  of the sequence 
$(|b_n-b|)$ of  nonnegative reals.

\section{Preliminaries on the proximal point algorithm}

In the sequel, $X$ is a CAT(0) space and  $f: X \to (-\infty, +\infty]$ is a convex, lower semicontinuous (lsc) 
proper function.  Let us recall that a minimizer of $f$ is a point $x\in X$ such that $f(x)=\inf_{y\in X} f(y)$. We denote  the set of minimizers of $f$ by $Argmin(f)$ ans we assume that $Argmin(f)$ is nonempty.  

The {\it proximal point mapping} or the {\em (Moreau-Yosida) resolvent}, as first introduced 
for CAT(0) spaces by Jost \cite{Jos95}, is a tool for finding minimizers of  such functions. 
For $\gamma>0$, the {\it resolvent} 
(or the {\it proximal mapping}) of $f$ of order $\gamma$ is the map $J^\gamma_f : X \to X$, 
defined, for any $x \in X$, by the following relation
$$\resgf(x) := {\arg\!\min}_{y \in X} \left[ \gamma f(y) + \frac12 d^2(x,y)\right].$$
This is the definition from \cite{Bac13},  as the factor of $2$ does not appear in the original 
paper of Jost, but this is, obviously, insignificant.
By \cite[Lemma 2]{Jos95},  the operator $\resgf$ is well-defined. We shall denote 
$\resuf$ simply by $J_f$. Then, for all $\gamma>0$ and for all $x\in X$,
\bua 
J_{\gamma f}(x) &=& \resgf(x)={\arg\!\min}_{y \in X} \left[ \gamma f(y) + 
\frac12 d^2(x,y)\right] \\
 &=&{\arg\!\min}_{y \in X} \left[ f(y) + \frac1{2\gamma} d^2(x,y)\right].
 \eua

The following property was also proved in \cite{Jos95}.

\begin{proposition}[{\cite[Lemma 4]{Jos95}}]\label{ne}
For any $\gamma >0 $, $J_{\gamma f}$ is nonexpansive, that is, for all $x,y\in X$,
\[d(J_{\gamma f}x,J_{\gamma f}y)\leq d(x,y).\]
\end{proposition}

We note that the definition of the proximal point mapping is motivated by the following proposition. 

\begin{proposition}\label{fix-min}
Let $x\in X$. Then  $x$ is a minimizer of $f$  if and only if  $x$ is a fixed point of $J_f$. 
\end{proposition}
\begin{proof}
Suppose first that $x$ is a minimizer of $f$. It follows that for all $y$,
$$f(x) + \frac12 d^2(x,x) =f(x)\leq f(y) \leq f(y) + \frac12 d^2(x,y),$$
therefore $x$ is also the argmin of the right hand side w.r.t. $y$ -- that is, $x=J_f(x)$.

Suppose now that $J_f(x)=x$. Then for all $y \in X$, as before,
$$f(x)\leq f(y) + \frac12 d^2(x,y).$$
Let $w \in X$. Using the fact that $f$ is convex, we get, for any $t \in (0,1)$,
\begin{eqnarray*}
f(x) &\leq & f((1-t)x+tw) + \frac12d^2(x,(1-t)x+tw) \\
&\leq & (1-t) f(x) + tf(w) +\frac12d^2(x,(1-t)x+tw)\\
&=& (1-t) f(x) + tf(w) + \frac12 t^2 d^2(x,w).
\end{eqnarray*}
Subtracting $(1-t)f(x)$ and dividing by $t$, we obtain that
$$f(x) \leq f(w) + \frac12 td^2(x,w),$$
and by letting $t \to 0$,  it follows that $f(x) \leq f(w)$. Since $w$ was chosen arbitrarily, 
we get that $x$ is a minimizer of $f$.
\end{proof}

Since, trivially, $Argmin(\gamma f)=Argmin(f)$, we get that 

\begin{corollary}\label{F-FixJgamma}
For any $\gamma > 0$, $Fix(J_{\gamma f}) = Fix(J_f)=Argmin(f)$.
\end{corollary}

We may now proceed to study the algorithm in itself.  Let  $(\gamma_n)_{n \in \mathbb{N}}$ 
be a sequence in $ (0, \infty)$.
The {\it proximal point algorithm} $(x_n)_{n\in\N}$ starting with $x\in X$ is defined as follows:
\[x_0 :=x, \qquad x_{n+1} := J_{\gamma_n f} x_n \text{~for all~} n\in\N.\]

Let us give some useful properties of the sequence $\sxn$. 

\begin{lemma}
For all $n ,m\in \mathbb{N}$ and all $p \in X$,
\begin{eqnarray}
d(x_{n+1},p) &\leq& d(x_n, p) + d(p, J_{\gamma_n f} p), \label{dxnunu-p}\\
d(x_{n+m},p) &\leq&  d(x_n, p) + \sum_{i=n}^{i=n+m-1} d(p, J_{\gamma_i f} p) \label{dxnm-p}.
\end{eqnarray}
\end{lemma}
\begin{proof}
We have that
\bua
 d(x_{n+1},p) &=& d(J_{\gamma_n f} x_n, p) \leq d(J_{\gamma_n f} x_n, J_{\gamma_n f} p) + 
d(J_{\gamma_n f} p, p) \\
&\leq & d(x_n, p) + d(p, J_{\gamma_n f} p).
\eua
\eqref{dxnm-p} follows immediately by induction on $m$.
\end{proof}

The following lemma contains results from \cite{Bac13}.

\begin{lemma}\label{fundam}
\be[(i)]
\item The sequence $\sfxn$ is nonincreasing.
\item For all $n\in\N$ and all $p\in Argmin(f)$,
\begin{eqnarray}
\!\!\!\!\!\! 2\gamma_n(f(x_{n+1}) - \min(f))&\leq & d^2(x_n,p)-d^2(x_{n+1},p)-d^2(x_n,x_{n+1}) \label{d2xnpp+1-Bac}\\
\!\!\!\!\!\! d^2(x_n,x_{n+1}) &\leq & d^2(x_n,p) - d^2(x_{n+1},p)  \label{d2xnpp+1}  \\
\!\!\!\!\!\! f(x_{n+1}) - \min(f) &\leq & \frac{d^2(x,p)}{2\sum_{i=0}^{n} \gamma_i}\label{fxnlambdanp}
\end{eqnarray}
\ee
\end{lemma}
\begin{proof}
\be[(i)]
\item This is used without proof in \cite{Bac13}, and hence we shall justify it. Let $n \in \mathbb{N}$.  
By the definition of $J_{\gamma_n f}$ and considering that $x_{n+1} = J_{\gamma_n f}x_n$, we have that:
$$ \gamma_n f(x_{n+1}) + \frac12 d^2(x_{n},x_{n+1}) \leq \gamma_nf(x_{n}) + \frac12 d^2(x_{n},x_{n}) = \gamma_nf(x_{n}),$$
and so,  $  \gamma_n f(x_{n+1}) \leq \gamma_nf(x_{n})$, 
hence $f(x_{n+1}) \leq f(x_n)$.
\item  With the assumption made in \cite{Bac13} that $\min(f)=0$, \eqref{d2xnpp+1-Bac} 
is the  last inequality in the proof of (7) from \cite{Bac13}, while \eqref{fxnlambdanp} is  obtained from the inequality before (8) 
in \cite{Bac13}.  Note also that $\lambda_k$ in \cite{Bac13}  corresponds to our $\gamma_{k-1}$. 
\eqref{d2xnpp+1}  follows immediately from \eqref{d2xnpp+1-Bac}.
\ee
\end{proof}

We finish this section with two effective results on the behaviour of the proximal point algorithm, 
results that will be also used in the next section to get our main quantitative theorem.

\begin{lemma}\label{rates}
Let $b\in \mathbb{R}$  be such that $d(x,p)\leq b$ for some $p\in Argmin(f)$.
\begin{enumerate}[(i)]
\item\label{rates-liminf}  $\displaystyle\liminf_{n \to \infty} d(x_n,x_{n+1}) = 0$ with modulus of liminf 
\beq
\Delta_b(k,L):=\lceil b^2(k+1)^2 \rceil + L - 1. \label{modliminf-xnxn+1}
\eeq
\item\label{rates-lim} Assume that $\sum_{n=0}^\infty \gamma_n = \infty$ with rate of divergence $\theta$. Then 
$\displaystyle\lim_{n \to \infty} f(x_n) = \min(f)$, with a (nondecreasing) rate of convergence
\beq
\beta_{b, \theta}(k):=\theta^M(\lceil b^2(k+1)/2 \rceil) +1.\label{rate-fxn}
\eeq

\end{enumerate}
\end{lemma}
\begin{proof}
\begin{enumerate}[(i)]
\item By \eqref{d2xnpp+1}, we get that for all $j\geq k$,
\bua
\sum_{n=j}^k d^2(x_n,x_{n+1}) &\leq&  \sum_{n=0}^k d^2(x_n,x_{n+1})\leq \sum_{n=0}^k (d^2(x_n,p) - d^2(x_{n+1},p))
 \\ &\leq& d^2(x,p) \leq   b^2.
\eua
Suppose that $d(x_n,x_{n+1}) > \frac1{k+1}$ for all $n \in [L,\Delta_b(k,L)]$. Then
$$(\Delta_b(k,L) - L + 1) \frac1{(k+1)^2} < \sum_{n=L}^{\Delta_b(k,L)} d^2(x_n,x_{n+1}) \leq b^2,$$
from which we get $\Delta_b(k,L) < b^2(k+1)^2 + L -1$, a contradiction. 
\item Since $f(x_n)\geq \min(f)$ 
for all $n\in\N$ and, by Lemma~\ref{fundam}.(i), $\sfxn$ is nonincreasing, all we have to show is that 
$$f(x_{\beta_{b, \theta}(k)}) - \min(f) \leq \frac1{k+1}.$$
Assume that this is is not true. Then, using \eqref{fxnlambdanp} and the fact that $\theta$ is a rate of divergence, 
we get that 
\bua
\frac1{k+1} &< & f(x_{\beta_{b, \theta}(k)})-\min(f)\leq 
\frac{b^2}{2\sum_{i=0}^{\theta^M(\lceil b^2(k+1)/2 \rceil)} \lambda_i} \\
&\leq & \frac{b^2}{2\sum_{i=0}^{\theta(\lceil b^2(k+1)/2 \rceil)} \lambda_i}  
\leq   \frac{b^2}{2\lceil b^2(k+1)/2\rceil}\leq
\frac1{k+1},
\eua
a contradiction.
\end{enumerate}
\end{proof}

\section{Quantitative results on the proximal point algorithm}

We will now proceed to derive the moduli that are needed in order to apply the results of \cite{KohLeuNicXX}.

As in the previous section, $f: X \to (-\infty, +\infty]$ is a convex, lsc proper function, and we set $F:=Argmin(f)\neq \emptyset$. 
For every $k\in \N$, let us define
\beq
AF_k := \left\{x \in X \mid \text{for all }i \leq k\text{, }d(x,J_{\gamma_i f}x) \leq \frac1{k+1} \right\}.
\eeq

\begin{proposition}
$(AF_k)$ is an approximation to $F$. 
\end{proposition}
\begin{proof}
Since, obviously, $(AF_k)$ is a nonincreasing sequence, it remains to prove that $\displaystyle F =  \bigcap_{k \in \mathbb{N}} AF_k$.

``$\subseteq$'' Let $x\in F$ and  $k\in \N$ be arbitrary. Then, for all $i\leq k$,  by Corollary \ref{F-FixJgamma},  
we have that $F=Fix(J_{\gamma_i f})$, hence, in particular, $d(x,J_{\gamma_i f}x) \leq \frac1{k+1}$.  Thus, $x\in AF_k$.

``$\supseteq$'' Let $x \in \bigcap_{k \in \mathbb{N}} AF_k$. It follows, in particular, that  for any $k\in\N$,
$$d(x,J_{\gamma_0 f}x) \leq \frac1{k+1},$$
As a consequence, we get that  $x \in Fix(J_{\gamma_0 f})=F$, again by Corollary \ref{F-FixJgamma}.
\end{proof}

This approximation will turn out to be convenient for the results we are aiming for.

\begin{proposition}
With respect to the above approximation, $F$ is uniformly closed with moduli 
\beq 
\delta_F(k):=2k+1, \quad \omega_F(k):=4k+3. \label{def-mod-uclosed}
\eeq
\end{proposition}
\begin{proof}
Let $k\in\N$ and $p,q\in X$ be such that $q \in AF_{2k+1}$ and $d(p,q) \leq \frac1{4k+4}$. We need to show that $p\in AF_k$, i.e. that for all $i \leq k$, $d(p,J_{\gamma_i f}p) \leq \frac1{k+1}$.
Let $i \leq k$ be arbitrary. We get that 
\begin{eqnarray*}
d(p,J_{\gamma_i f}p) &\leq &  d(p,q) + d(q,J_{\gamma_i f}q) + d(J_{\gamma_i f}q,J_{\gamma_i f}p) \leq 2d(p,q) + d(q,J_{\gamma_i f}q)  \\
&\leq & \frac2{4k+4} + \frac1{2k+2} =  \frac1{k+1},
\end{eqnarray*}
where we have used at the second inequality the fact that $J_{\gamma_i f}$ is nonexpansive.
\end{proof}

\begin{lemma}
The sequence $\sxn$ is uniformly Fej\'er monotone w.r.t. $(AF_k)$ with modulus 
\begin{equation}
\chi(n,m,r):=\max\{n+m-1,m(r+1)\}. \label{mod-unif-fej}
\end{equation}
\end{lemma}
\begin{proof}
Let $n,m,r \in \mathbb{N}$, $p \in AF_{\chi(n,m,r)}$ and $l \leq m$. We get that 
\begin{eqnarray*}
d(x_{n+l},p) & \leq &  d(x_n, p) + \sum_{i=n}^{i=n+l-1} d(p, J_{\gamma_i f} p)\quad \text{by \eqref{dxnm-p}}\\
& \leq &  d(x_n, p) + \sum_{i=n}^{i=n+m-1} d(p, J_{\gamma_i f} p)\\
& \leq & d(x_n, p) + \frac{m}{\chi(n,m,r)+1} \\
& < &d(x_n, p) +\frac1{r+1},
\end{eqnarray*}
where at the second-to-last inequality we used that $\chi(n,m,r) \geq n+m-1$, so $d(p, J_{\gamma_i f} p)\leq \frac1{\chi(n,m,r)+1}$ for all $i=n,\ldots, n+m-1$, and at the last one, that $\chi(n,m,r) \geq m(r+1)$.
\end{proof}

\begin{proposition}\label{F-approximate-points}
Let $b\in \mathbb{R}$  be such that $d(x,p)\leq b$ for some $p\in F$ and assume that $\sum_{n=0}^\infty \gamma_n = \infty$ with 
rate of divergence $\theta$. Suppose, moreover, that $M:\N\to (0,\infty)$ is such that $M(k)\geq \max_{0 \leq i \leq k} \gamma _i$ for all 
$k\in\N$. 
Then  $\sxn$ has approximate $F$-points w.r.t. $(AF_k)$ with a (nondecreasing) modulus
\beq
\Phi_{b,\theta,M}(k):=\lceil b^2(k+1)^2\rceil + \beta_{b,\theta}(\lceil 2(k+1)^2M(k) \rceil-1) \label{def-Phi},
\eeq
where $\beta_{b,\theta}$ is defined by \eqref{rate-fxn}.
\end{proposition}
\begin{proof}
Let $k \in \mathbb{N}$ be arbitrary. Denote, for simplicity, $$c:=\beta_{b,\theta}(\lceil 2(k+1)^2M(k) \rceil-1).$$
Applying Lemma~\ref{rates}.\eqref{rates-liminf}, we obtain that there is an $N \in [c, \Delta_b(k,c)]$ such that
$$d(x_{N}, x_{N +1}) \leq \frac1{k +1},$$
where $\Delta_b$ is given by \eqref{modliminf-xnxn+1}. We remark, first, that
$$N \leq \Delta_b(k,c) = \lceil b^2(k+1)^2\rceil + c - 1 < \lceil b^2(k+1)^2\rceil + c = \Phi_{b,\theta,M}(k).$$
Since $N\geq c$ and $\beta_{b,\theta}$ is a rate of convergence of $\sfxn$ towards $\min(f)$
(by Lemma~\ref{rates}.\eqref{rates-lim}),  we get that
$$ f(x_N) \leq \min(f)+\frac1{\lceil 2(k+1)^2M(k) \rceil}.$$
On the other hand, for all $i \leq k$, we have, by the definition of $J_{\gamma_i f}$, that
$$ \gamma_i f(J_{\gamma_i f} x_{N}) + \frac12 d^2(x_{N},J_{\gamma_i f} x_{N}) \leq \gamma_if(x_{N}) + \frac12 d^2(x_{N},x_{N}) = \gamma_if(x_{N}).$$
As $f(J_{\gamma_i f} x_{N}) \geq \min(f)$, it follows that, for all $i \leq k$,
\bua d^2(x_{N},J_{\gamma_i f} x_{N}) &\leq & 2\gamma_i(f(x_{N})-f(J_{\gamma_i f} x_{N}))\leq 
2\gamma_i(f(x_{N})-\min(f))\\
&\leq & 2M(k) \frac1{\lceil 2(k+1)^2M(k) \rceil} \leq \frac1{(k+1)^2}.
\eua
Thus, we have proved that for all $k \in \mathbb{N}$ there is an $N \leq \Phi_{b,\theta,M}(k)$ such that for all $i \leq k$,
$$d(x_N, J_{\gamma_i f} x_N) \leq \frac1{k+1}.$$
That's what was required.
\end{proof}

Now that all the necessary moduli have been computed, we may apply \cite[Theorems 5.1 and 5.3]{KohLeuNicXX} to get our main result, which finitarily expresses the strong convergence of the proximal point algorithm to a minimizer of $f$.

\begin{theorem}\label{main-quant-thm}
Let $b>0$, $\alpha,\theta:\N\to \N$ and $M:\N\to(0,\infty)$. Define  
$\Psi_{b,\theta,M,\alpha}, \, \Omega_{b,\theta,M,\alpha}:\N\times\N^\N\to\N$ as in Table \ref{tabel-1}.
Then for all
\be[(i)]
\item totally bounded CAT(0) spaces with modulus of total boundedness $\alpha$;
\item convex lsc proper mappings  $f : X \to (-\infty,\infty]$ with $Argmin(f)\ne\emptyset$;
\item  $x\in X$ such that $d(x,p)\leq b$ for some minimizer $p$ of $f$;
\item  sequences $(\gamma_n)$ in $(0,\infty)$ such that $\sum_{n=0}^\infty \gamma_n = \infty$ with 
rate of divergence $\theta$ and $M(k)\geq \max_{0 \leq i \leq k} \gamma _i$ for all $k\in\N$;
\ee
we have that
\be[(i)] 
\item $\Psi_{b,\theta,M,\alpha}$ is a rate of metastability for  the proximal point algorithm $\sxn$ starting with $x$, 
i.e. for all $k \in \mathbb{N}$ and  all $g : \mathbb{N} \to \mathbb{N}$ there is an $N \leq \Psi_{b,\theta,M,\alpha}(k,g)$ such that for all $i,j \in [N,N+g(N)]$,
$$d(x_i,x_j) \leq \frac1{k+1}.$$
\item For all $k \in \mathbb{N}$ and all $g : \mathbb{N} \to \mathbb{N}$ there is an $N \leq \Omega_{b,\theta,M,\alpha}(k,g)$ 
such that for all $i,j \in [N,N+g(N)]$,
$$d(x_i,x_j) \leq \frac1{k+1}$$ and for all $i\in[N,N+g(N)]$ and all $d \leq k$,
$$d(x_i, J_{\gamma_d f}x_i) \leq \frac1{k+1}.$$
\ee
\end{theorem}
\begin{proof}
Apply \cite[Theorem 5.1]{KohLeuNicXX} to get (i) and  \cite[Theorem 5.3]{KohLeuNicXX} to obtain (ii).
Remark that in our case, using the notations 
from \cite{KohLeuNicXX}, $\alpha_G=\beta_H=id_{\R_+}$, hence $P=\alpha(4k+3)$ and,  furthermore,
\bua
\chi_g^M(n,r)&=&\max_{i \leq n} \chi_g(i,r)=\max_{i \leq n}\chi(i,g(i),r)\\
&=& \max_{i \leq n} \max\{i+g(i)-1,g(i)(r+1)\}, \text{~by \eqref{mod-unif-fej}}\\
k_0 &=&\max\left\{ k,\left\lceil\frac{\omega_F(k)-1}{2}\right\rceil \right\}=2k+1, \text{~by \eqref{def-mod-uclosed}}\\
(\chi_{k, \delta_F})^M_g(n,r) &=& \max\{2k+1,\max_{i \leq n}\max\{i+g(i)-1,g(i)(r+1)\}\}.
\eua
We denote, for simplicity, $\chi_{k, \delta_F}$ by $\tilchi$.
\end{proof}

\begin{table}[ht!]
\begin{center}
\scalebox{0.95}{\begin{tabular}{ | l |}
\hline\\[1mm]
$\Psi_{b,\theta,M,\alpha}(k,g):=(\Psi_0)_{b,\theta,M}(\alpha(4k+3),k,g)$\\[2mm]
$(\Psi_0)_{b,\theta,M}(0,k,g):=0$ \\[2mm]
$(\Psi_0)_{b,\theta,M}(n+1,k,g):=
\Phi_{b,\theta,M}\big(\chi_g^M((\Psi_0)_{b,\theta,M}(n,k,g),4k+3)\big)$\\[2mm]
$\chi_g^M(n,r)=\max_{i \leq n} \max\{i+g(i)-1,g(i)(r+1)\}$\\[4mm]
$\Omega_{b,\theta,M,\alpha}(k,g):= (\Omega_0)_{b,\theta,M}(\alpha(8k+7),k,g)$\\[2mm]
$(\Omega_0)_{b,\theta,M}(0,k,g):=0$\\[2mm]
$(\Omega_0)_{b,\theta,M}(n+1,k,g):=
\Phi_{b,\theta,M}\big(\tilchi^M_g((\Omega_0)_{b,\theta,M}(n,k,g),8k+8\big)$\\[2mm]
$\tilchi^M_g(n,r)=\max\{2k+1,\max_{i \leq n}\max\{i+g(i)-1,g(i)(r+1)\}\}$\\[2mm]
with $\Phi_{b,\theta,M}$ given by \eqref{def-Phi}\\[1mm]
\hline 
\end{tabular}}\\[1mm]
\caption{Definitions of $\Psi_{b,\theta,M,\alpha}$ and $\Omega_{b,\theta,M,\alpha}$} \label{tabel-1}
\end{center}
\end{table}

The above theorem can be considered a ``true" finitization (in the sense of Tao) of  Theorem \ref{ppa-cat0-strong}, since 
\begin{enumerate}[(i)]
\item it involves only a finite segment of the proximal point algorithm $(x_n)$; 
\item the existence of a rate of metastability  $\Psi_{b,\theta,M,\alpha}$ is, as previously stated, classically equivalent 
to Cauchyness; 
\item  the existence of the second rate $\Omega_{b,\theta,M,\alpha}$ guarantees, for complete CAT(0) spaces,  
that the limit of the sequence is an element of $F=Argmin(f)$ (see \cite[Remark 5.5]{KohLeuNicXX}); 
\item the modulus of total boundedness only needs to apply to the ball of radius $b$ considered in the proof, 
therefore we have derived strong convergence for locally compact CAT(0) spaces, as pointed out in \cite[Remark 5.4]{KohLeuNicXX}.
\end{enumerate} 
Furthermore, both rates $\Psi_{b,\theta,M,\alpha}$ and $\Omega_{b,\theta,M,\alpha}$ are computable and are, moreover, expressed using primitive 
recursive functionals.


\begin{thebibliography}{}
\bibitem{Bac13}
M. Ba\v{c}\'ak,
The proximal point algorithm in metric spaces,
Israel J. Math. 194 (2013),  689--701.

\bibitem{BauCom10}
H. Bauschke, P. Combettes,
{\it Convex Analysis and Monotone Operator Theory in Hilbert Spaces},
Springer, 2010.

\bibitem{BreLio78}
H. Br\'ezis, P. Lions,
Produits infinis de r\'esolvantes, 
Israel J. Math. 29 (1978), 329--345.

\bibitem{Ger08} 
P. Gerhardy, 
Proof mining in topological dynamics, 
Notre Dame J. Form. Log. 49 (2008), 431--446.


\bibitem{GerKoh08}
P. Gerhardy, U. Kohlenbach,
General logical metatheorems for functional analysis, 
Trans. Amer. Math. Soc. 360 (2008), 2615--2660.

\bibitem{Goe58}
K. G\"odel, 
\"Uber eine bisher noch nicht ben\"utzte Erweiterung des finiten Standpunktes,
 Dialectica 12 (1958), 280--287. 
 
\bibitem{GunKoh16}
D. G\"unzel, U. Kohlenbach, 
Logical metatheorems for abstract spaces axiomatized in positive bounded logic, 
Adv. Math. 290 (2016), 503-551.

\bibitem{Jos95}
J. Jost,
Convex functionals and generalized harmonic maps into spaces of non positive curvature,
Comment. Math. Helv. 70 (1995), 659--673.


\bibitem{Koh05}
U. Kohlenbach,
Some logical metatheorems with applications in functional analysis, 
Trans. Amer. Math. Soc. 357 (2005), 89--128.

\bibitem{Koh08}
U. Kohlenbach,
{\it Applied proof theory: Proof interpretations and their use in mathematics},
Springer Monographs in Mathematics, Springer, 2008.

\bibitem{Koh16}
U. Kohlenbach,
Recent progress in proof mining in nonlinear analysis, 
to appear in IFCoLog Journal of Logic and its Applications.
Special issue with invited articles by recipients of a G\"odel Centenary Research Prize Fellowship.

\bibitem{KohLeu12}
U. Kohlenbach, L. Leu\c stean, 
On the computational content of convergence proofs via Banach limits, 
Philos. Trans. R. Soc. Lond. Ser. A Math. Phys. Eng. Sci.   A 370 (2012), 3449--3463.


\bibitem{KohLeuNicXX}
U. Kohlenbach, L. Leu\c stean, A. Nicolae,
Quantitative results on Fej\'er monotone sequences, 
arXiv:1412.5563 [math.LO], 2015, to appear in Commun. in Contemp. Math..

\bibitem{KohLopNic17}
U. Kohlenbach, G. L\'opez-Acedo, A. Nicolae,
Quantitative asymptotic regularity for the composition of two mappings, 
Optimization 66 (2017), 1291--1299. 

\bibitem{KohNic17}
U. Kohlenbach, A. Nicolae, 
A proof-theoretic bound extraction theorem for CAT($\kappa$) spaces, 
Studia Logica 105 (2017),  611-624.

\bibitem{Leu06}
L. Leu\c stean, 
Proof mining in $\R$-trees and hyperbolic spaces, 
Electron. Notes Theor. Comput. Sci (Proceedings of WoLLIC 2006) 165 (2006), 95-106.

\bibitem{LeuRadSip16}
L. Leu\c stean, V. Radu, A. Sipo\c s,
Quantitative results on the Ishikawa iteration of Lipschitz pseudo-contractions, 
J. Nonlinear Convex Anal. 17 (2016), 2277-2292.

\bibitem{Luc89}
H. Luckhardt, 
Herbrand-Analysen zweier Beweise des Satzes von Roth: Polynomiale Anzahlschranken, 
J. Symbolic Logic 54 (1989),  234--263.

\bibitem{Mar70}
B. Martinet,
R\'egularisation d'in\'equations variationnelles par approximations successives,
Rev. Fran\c caise Informat. Recherche Op\'erationnelle 4 (1970), 154--158.

\bibitem{Roc76}
T. Rockafellar,
Monotone operators and the proximal point algorithm,
SIAM J. Control Optim. 14 (1976), 877--898.

\bibitem{SipXX}
A. Sipo\c s, 
Proof mining in $L^p$ spaces,  
arXiv:1609.02080 [math.LO], 2016, submitted.

\bibitem{Tao07} 
T. Tao,
Soft analysis, hard analysis, and the finite convergence principle,
Essay posted May 23, 2007, appeared in: 
T. Tao,
{\it Structure and Randomness: Pages from Year One of a Mathematical Blog},
Amer. Math. Soc., 2008.

\bibitem{Tao08} 
T. Tao,
Norm convergence of multiple ergodic averages for commuting transformations,
Ergodic Theory Dynam. Systems 28 (2008), 657--688. 
\end{thebibliography}
\end{document}